\documentclass[11pt,final]{article}
\usepackage{amsmath,amsthm,amsfonts,amssymb,graphicx,hyperref,enumerate,psfrag}
\usepackage{fullpage}

\usepackage[utf8]{inputenc} 

\newtheorem{lemma}{Lemma}

\newtheorem{remark}{Remark}
\newtheorem{proposition}{Proposition}
\newtheorem{theorem}{Theorem}

\newtheorem{corollary}{Corollary}

\newcommand{\EE}{{\mathbb{E}}}

\newcommand{\dd}{\,{\rm d}}
\newcommand{\PP}{\mathbb{P}}
\newcommand{\Z}{\mathbb {Z}}
\newcommand{\N}{\mathbb {N}}
\newcommand{\R}{\mathbb {R}}

\newcommand{\cI}{\mathcal {I}}

\newcommand{\cO}{\mathcal {O}}

\newcommand{\cL}{\mathcal {L}}
\newcommand{\cP}{\mathcal {P}}

\newcommand{\p}{\mathfrak{q}}

\newcommand{\bX}{{\bf X}}

\newcommand{\dtv}{d_{\textsc{tv}}}

\newcommand{\q}{\overline{{\mathfrak q}}}

\newcommand{\tmix}{t_{\textsc{mix}}}

\newcommand{\bI}{{\bf{I}}}
\newcommand{\bJ}{{\bf{J}}}

\title{Cutoff for the mean-field zero-range process with bounded monotone rates}
\author{Jonathan Hermon, Justin Salez}

\begin{document}
\maketitle
\begin{abstract}
We consider the zero-range process with arbitrary bounded monotone rates on the complete graph, in the regime where the number of sites diverges while the density of particles per site converges. We determine the asymptotics of the mixing time from any initial configuration, and establish the cutoff phenomenon.  The intuitive picture is that the system separates into a slowly evolving solid phase and a quickly relaxing liquid phase: as time passes, the solid phase   dissolves  into the liquid phase, and the mixing time is essentially the time at which the system becomes completely liquid. Our proof uses the path coupling technique of Bubley and Dyer, and the analysis of a suitable hydrodynamic limit. To the best of our knowledge, even the order of magnitude of the mixing time was unknown, except in the special case of constant rates.
\end{abstract}
\tableofcontents
\newpage
\section{Introduction}
\subsection{Model}
Introduced by Spitzer in 1970 \cite{Spitzer}, the \emph{zero-range process} is a widely studied model of interacting random walks, see, e.g. \cite{liggettbook2,liggettbook1,Evans} and the references therein. It describes the  evolution of $m\ge 1$ indistinguishable particles randomly hopping across $n\ge 1$ sites. The interaction is specified  by a function $r\colon \{1,2,\ldots\}\to(0,\infty)$, where $r(k)$ indicates the rate at which particles are expelled from a site when $k$ particles are present on it. We will here focus on the \emph{mean-field} version of the model, where all jump destinations are uniformly distributed. More formally, we consider a continuous-time Markov chain ${\bf X}:=(X(t)\colon t\ge 0)$ taking values in the state space 
\begin{eqnarray}
\Omega & := & \left\{x=(x_1,\ldots,x_n)\in\Z_+^n\colon \sum_{i=1}^nx_i=m\right\},
\end{eqnarray}
and whose Markov generator $\cL$ acts on observables $\varphi\colon \Omega\to\R$ as follows:
\begin{eqnarray}
\label{def:markov}
(\cL \varphi)(x) & = & \frac{1}{n}\sum_{1\le i,j\le n}r(x_i)\left(\varphi(x+\delta_j-\delta_i)-\varphi(x)\right).
\end{eqnarray}
Here $(\delta_i)_{1\le i\le n}$ denotes the canonical basis of $\Z^n_+$, and we adopt the convention that $r(0)=0$ (no jumps from empty sites). The generator $\cL$ is clearly irreducible, with reversible law
\begin{eqnarray}
\label{statio}
\pi(x) & \propto & \prod_{i=1}^n\prod_{k=1}^{x_i}\frac{1}{r(k)}.
\end{eqnarray}
The present paper is concerned with the problem of estimating the speed at which the  convergence to equilibrium occurs, as quantified by the so-called \emph{mixing times}:
\begin{eqnarray}
\label{def:tmix}\tmix(x;\varepsilon) & := & \min\left\{t\ge 0\colon \left\|\PP_x\left(X(t)\in \cdot\right)-\pi\right\|_{\textsc{tv}}\le \varepsilon\right\}.
\end{eqnarray}
In this definition, $\|\mu-\nu\|_\textsc{tv}=\max_{A\subseteq \Omega}\left|\mu(A)-\nu(A)\right|$ denotes the total-variation distance, and the parameters $x\in\Omega$ and $\varepsilon\in(0,1)$ specify the initial state  and the desired precision, respectively. Of particular interest is the \emph{worst-case} mixing time, obtained by maximizing over all initial states:
\begin{eqnarray}
\label{def:tmixworst}\tmix(\varepsilon) & := & \max\left\{\tmix(x;\varepsilon)\colon {x\in\Omega}\right\}.
\end{eqnarray}
Understanding this fundamental parameter -- and in particular, its dependency in the precision $\varepsilon\in(0,1)$ -- is in general a challenging task, see the books \cite{MR2341319,MR3726904} for a comprehensive account.  Our current knowledge on the total-variation mixing time of the zero-range process is embarrassingly limited in comparison with the numerous functional-analytic estimates that have been established over the past decades \cite{MR1415232,MR1681098,MR2073330,MR2184099,MR2200172,MorrisZRP,Cap,Graham,MR3513606,HS}. In fact, to the best of our knowledge, the exact order of magnitude of the mixing time of the zero-range process has only been determined in the very special case where the rate function $r$ is constant \cite{Lacoin,Lacoincycle,MS,HS}.

\subsection{Main result}The rate function $r$ will remain fixed throughout the paper, and will only be assumed to be non-decreasing and bounded. Upon re-scaling time by a constant factor if necessary, we take
\begin{eqnarray}
\label{assume:rates}
\lim_{k\to\infty}\uparrow r(k) & = & 1.
\end{eqnarray}
Our results will be expressed in terms of a certain generating series associated with $r$, namely
\begin{eqnarray}
\label{def:psi}
\Psi(z):=\frac{zR'(z)}{R(z)}, & \textrm{ where } & R(z) := \sum_{k=0}^\infty\frac{z^k}{r(1)\cdots r(k)}.
\end{eqnarray}
All asymptotic statements will refer to the regime where the density of particles per site stabilizes:
\begin{eqnarray}
\label{assume:sparse}
n\to\infty, & \qquad & \frac{m}{n}\to \rho\in[0,\infty).
\end{eqnarray}
To lighten the notation, we will keep the dependency upon $n$ implicit as often as possible. Our main result is the following explicit asymptotics for the worst-case mixing time:
\begin{theorem}[Worst-case mixing time]\label{th:main}For any fixed $\varepsilon\in(0,1)$, 
\begin{eqnarray}
\label{eq:main}
\frac{\tmix(\varepsilon)}{n} & \xrightarrow[n\to\infty]{} & \gamma:=\int_0^\rho\frac{{\rm d}s}{1-\Psi^{-1}(s)},
\end{eqnarray}
where $\Psi^{-1}\colon[0,\infty)\to[0,1)$ denotes the inverse of the (increasing) bijection $\Psi\colon[0,1)\to[0,\infty)$. 
\end{theorem} 
Although it seems intuitively clear that the worst-case mixing time should be achieved by initially placing all particles on the same site, there does not appear to be any direct  justification of this fact. We will thus determine the asymptotics of the mixing time from \emph{every} possible configuration $x\in\Omega$, see Corollary \ref{co:pro} below for the detailed result. The notable disappearance of $\varepsilon$ on the right-hand side of (\ref{eq:main}) reveals a sharp transition in the convergence to equilibrium of the process, known as a \emph{cutoff} \cite{MR770418,MR1374011}. To the best of our knowledge, the occurence of this phenomenon for the mean-field zero-range process was only known in the  special \emph{rate-one} case, where the function $r$ is simply constant equal to $1$ \cite{MS}. This choice trivially fits our setting (\ref{assume:rates}), with
\begin{equation}
R(z)=\frac{1}{1-z},\qquad \Psi(z)=\frac{z}{1-z}, \qquad \Psi^{-1}(s)=\frac{s}{1+s},\qquad \gamma=\rho+\frac{\rho^2}{2}.
\end{equation}
Beyond the obvious complications raised by the non-explicit nature of the rates, Theorem \ref{th:main} requires new ideas for at least two reasons. First, the crucial spectral gap estimate of Morris  \cite{MorrisZRP}, on which the whole argument of \cite{MS} ultimately relies, is only available in the rate-one case.  Second, the stationary distribution (\ref{statio}) is no longer uniform, making the entropy computations from \cite{Graham,MS} unapplicable. As a result, even the order of magnitude $\tmix(\varepsilon)=\Theta(n)$ appears to be new. We circumvent these obstacles by resorting to the powerful \emph{path coupling} method of Bubley and Dyer \cite{Pathcoupling}. This alternative route turns out to be so efficient that the proof of our generalization ends up being significantly shorter than that of the original result \cite{MS}, without using anything from it. 

\subsection{Proof outline}

Intuitively, the system may be viewed as consisting of two regions evolving on different time-scales:
\begin{itemize}
\item a slow \emph{solid phase}, consisting of those sites which are occupied by $\Theta(n)$ particles
\item a quick \emph{liquid phase}, formed by those sites which are occupied by $o(n)$ particles.
\end{itemize}
The presence of a solid phase is a clear indication that the system is out of equilibrium, since under the stationary law $\pi$, the maximum occupancy is easily seen to be $\Theta(\log n)$ (see, e.g., (\ref{XXX}) below). What is less obvious, but true, is that conversely, any completely liquid system reaches equilibrium in neglible time. To make this rigorous, we use the path coupling method of Bubley and Dyer \cite{Pathcoupling}. Note that in the regime  (\ref{assume:sparse}), there is $\overline{\rho}<\infty$, independent of $n$, such that
\begin{eqnarray}
\label{assume:bounded}
\frac{m}{n} & \le & \overline{\rho}.
\end{eqnarray}
By a \emph{dimension-free constant}, we will always mean a number which depends only on $\overline{\rho}$ and $r$.

\begin{theorem}[Fast mixing]\label{th:liquid} There is a dimension-free constant $\kappa<\infty$ such that
\begin{eqnarray}
\label{assume:liquid}
\tmix(x;\varepsilon)  & \le & \kappa \|x\|_\infty+(\ln n)^{\kappa},
\end{eqnarray}
for every $x\in\Omega$ and every $\varepsilon\in(0,1)$, provided $n\ge \kappa/\varepsilon$.
\end{theorem}
When combined with the worst-case bound $\|x\|_\infty\le m$, this already yields the correct order of magnitude $\tmix(\varepsilon)  = \cO(n)$, for any fixed $\varepsilon\in(0,1)$. However, the real interest of Theorem \ref{th:liquid} lies in the linear dependency in $\|x\|_\infty$, which implies that the equilibrium is attained in negligible time when the initial configuration $x=x(n)$ is completely liquid: for any fixed $\varepsilon\in(0,1)$, we have
\begin{eqnarray}
\label{co:liquid}
\|x\|_\infty=o(n) & \Longrightarrow & \tmix(x;\varepsilon) = o(n).
\end{eqnarray}
By the Markov property, this reduces our task to that of understanding the time it takes for an arbitrary initial condition $x\in\Omega$ to become completely liquid.  By symmetry, we may assume without loss of generality that the coordinates of $x$ are  arranged in decreasing order:
\begin{eqnarray}
\label{assume:order}
x_1 \ \ge & \ldots & \ge \ x_n.
\end{eqnarray}
In the regime (\ref{assume:sparse}), we may further assume (upon passing to a subsequence) that for each $k\ge 1$, 
\begin{eqnarray}
\label{assume:profile}
\frac{x_k}{n} & \xrightarrow[n\to\infty]{} & u_k.
\end{eqnarray}
Note that the limiting profile $(u_k)_{k\ge 1}$ must then  satisfy 
$
u_1\ge u_2\ge\ldots\ge 0$  and $\sum_{k=1}^\infty u_k \le \rho$.
In this setting, our second step will consist in establishing a deterministic approximation of the form 
\begin{eqnarray}
\frac{X_k(nt)}{n} & \approx & \left[u_k-f(t)\right]_+,
\end{eqnarray}
for fixed $t\ge 0$ and $k\ge 1$, where  $f\colon\R_+\to\R_+$ is a smooth increasing function describing the dissolution of the solid phase, and where we have used the notation  $[a]_+=\max(a,0)$.
\begin{theorem}[Hydrodynamic limit]
\label{th:solid}
If $x=x(n)$ satisfies (\ref{assume:order})-(\ref{assume:profile}), then for  any fixed $T\ge 0$,
\begin{eqnarray}
\EE_x\left[\sup_{k\in[n],t\in[0,T]}\left|\frac{X_{k}(nt)}{n}-\left[u_k-f(t)\right]_+\right|\right] & \xrightarrow[n\to\infty]{} &  0,
\end{eqnarray}
 where the function $f\colon\R_+\to\R_+$ is characterized by the differential equation
\begin{eqnarray}
\label{def:hydro}
f'(t) & = & 1-\Psi^{-1}\left(\rho-\sum_{k=1}^\infty \left[u_k-f(t)\right]_+\right),\qquad f(0)=0.
\end{eqnarray}  
Moreover, for each $i\ge 1$, we have the explicit expression
\begin{eqnarray}
f^{-1}\left(u_i\right) & = & \sum_{k=i}^\infty\frac{1}{k}\int_{\rho_k}^{\rho_{k-1}}\frac{{\rm d}s}{1-\Psi^{-1}(s)},
 \end{eqnarray} 
 where the numbers $\rho=\rho_0\ge\rho_1\ge\ldots\ge 0$ are given by $\rho_k := \rho+ku_{k+1}-\sum_{i=1}^ku_i.$
\end{theorem}
The proof relies on a  \emph{separation of timescales} argument: the liquid phase relaxes so quickly that, on the relevant time-scale, the solid phase can be considered as inert. Consequently, the liquid phase is permanently maintained in a {metastable} state resembling the true equilibrium, except that the density is lower because a macroscopic number of particles are  ``stuck'' in the solid phase. This imposes a simple asymptotic relation between  the number of particles in the solid phase and the dissolution rate, from which the autonomous  equation (\ref{def:hydro}) arises. Note that this limiting description gives access to the dissolution time of the system: for any $t\ge 0$, we have
\begin{eqnarray}
\label{co:solid}
{\|X(nt)\|_\infty}=o(n) & \Longleftrightarrow & t\ge f^{-1}(u_1).
\end{eqnarray}
Combining this with (\ref{co:liquid}), we readily obtain the following complete description of the mixing time.
\begin{corollary}[Mixing time from any state]\label{co:pro} In the regime (\ref{assume:order})-(\ref{assume:profile}), we have for fixed $\varepsilon\in(0,1)$,
\begin{eqnarray}
\label{eq:pro}
\frac{\tmix(x;\varepsilon)}{n} & \xrightarrow[n\to\infty]{} & f^{-1}(u_1)=\sum_{k=1}^\infty\frac{1}{k}\int_{\rho_k}^{\rho_{k-1}}\frac{{\rm d}s}{1-\Psi^{-1}(s)}.
\end{eqnarray}
\end{corollary}
Note that in the degenerate case where $u_1=0$, the right-hand side of (\ref{eq:pro}) vanishes, in agreement with (\ref{co:liquid}).
From this detailed description, the worst-case mixing time can finally be extracted by maximizing the right-hand side of (\ref{eq:pro}) over all possible profiles $(u_k)_{k\ge 1}$: we trivially always have
\begin{eqnarray}
 \sum_{k=1}^\infty\frac{1}{k}\int_{\rho_k}^{\rho_{k-1}}\frac{{\rm d}s}{1-\Psi^{-1}(s)} & \le & \int_{0}^{\rho}\frac{{\rm d}s}{1-\Psi^{-1}(s)},
\end{eqnarray}
and the equality is moreover attained when $u_1=\rho$ and $u_2=u_3=\ldots=0$. The maximizer corresponds to placing all   particles on the same site, as anticipated. This clearly establishes Theorem \ref{th:main}, and the remainder of the paper is  devoted to the proofs of Theorems \ref{th:liquid} and \ref{th:solid}, in Sections \ref{sec:liquid} and \ref{sec:solid}.

\section{Fast mixing in the absence of a solid phase}
\label{sec:liquid}
\subsection{Preliminaries}
We start by enumerating a few standard facts that will be used repeatedly in the sequel. 
\paragraph{Graphical construction.}\label{rk:graphical}Let $\Xi$ be a Poisson point process  of intensity $\frac 1n\, {\rm d}t\otimes {\rm d}u\otimes\rm{Card}\otimes\rm{Card}$ on $[0,\infty)\times [0,1]\times[n]\times[n]$ (where $\rm{Card}$ denotes the counting measure), and consider the piece-wise constant process $\bX=(X(t)\colon t\ge 0)$ defined by the initial condition $X(0)=x$ and the following jumps: for each point $(t,u,i,j)\in\Xi$,
\begin{eqnarray}
\label{jump}
X(t) & := & 
\left\{
\begin{array}{ll}
X(t-)+\delta_j-\delta_i & \textrm{if }  r\left(X_i(t-)\right)\ge u\\
X(t-) & \textrm{otherwise}.
\end{array}
\right.
\end{eqnarray}
Then ${\bf X}$ is a Markov process with generator $\cL$ and initial state $x$. We always use this construction. 

\paragraph{Monotony.} 
Since $r$ is non-decreasing, the above construction provides a \emph{monotone coupling} of trajectories: if we start from two configurations $x,y\in\Z_+^n$ satisfying  $x\le y$ (coordinate-wise), then this property is preserved by the jumps (\ref{jump}), so the resulting processes $\bf X,\bf Y$ satisfy 
\begin{eqnarray}
\label{monotone}
\forall t\ge 0,\qquad {X}(t) & \le & { Y}(t).
\end{eqnarray}
This classical fact will play an important role in our proof.

\paragraph{Stochastic regularity} For any $i\in[n]$ and $0\le s\le t$, we  have by construction
\begin{eqnarray}
\label{tightness}
-\Xi\left(\left[s,t\right]\times [0,1]\times\{i\}\times [n]\right) & \le \ X_i(t)-X_i(s) \ \le &   \Xi\left(\left[s,t\right]\times [0,1]\times [n]\times \{i\}\right), 
\end{eqnarray}
and the Poisson random variables $\Xi(\cdot)$ appearing on both sides have mean $t-s$. 

\paragraph{Mean-field jump rate.} For any observable $\varphi\colon\Omega\to\R$, the  process ${\bf M}=\left(M(t)\colon t\ge 0\right)$ given by
\begin{eqnarray}
\label{dynkin}
M(t) & := & \varphi\left(X(t)\right)-\varphi\left(x\right)-\int_0^t(\cL \varphi)\left(X(s)\right)\dd s  
\end{eqnarray}
is a zero-mean martingale, see e.g. \cite{MR838085}. In particular, when $\varphi(y)=y_i$ ($i\in[n]$), we obtain
\begin{eqnarray}
\label{dynkin:single}
M(t) & = & X_i(t)-x+\int_0^t r\left(X_i(s)\right)\dd s  - \int_0^t\zeta(s)\dd s, 
\end{eqnarray}
where the \emph{mean-field jump rate} $\zeta(t)$ is defined by 
\begin{eqnarray}
\label{def:zeta}
\zeta(t) & := & \frac 1n\sum_{j=1}^nr\left(X_j(t)\right).
\end{eqnarray}
Understanding the evolution of $\left(\zeta(t)\colon t\ge 0\right)$ will constitute  an important step in the proof. 

\subsection{Uniform downward drift}
Our first task consists in showing that the number of particles on a site can not stay large for long. 

\begin{proposition}[Uniform downward drift]\label{pr:uniform}
There are dimension-free constants $\theta,\delta>0$ such that
\begin{eqnarray}
\label{eq:uniform}
\EE_x\left[e^{\theta X_i(t) }\right] & \le & 2\left(1+e^{\theta (x_i-\delta t)}\right),
\end{eqnarray}
for all $x\in\Omega$, $i\in[n]$ and $t\in\R_+$. In particular, for any $a\ge 0$,
\begin{eqnarray}
\label{XXX}
\PP_x\left(X_i(t)\ge \left[x_i-\delta t\right]_++a\right) & \le & 4e^{-\theta a}.
\end{eqnarray}
\end{proposition}
  The intuition behind this result is as follows: if $X_i(t)$ is large, then by (\ref{assume:rates}) and (\ref{dynkin:single}), the drift of  $X_i(t)$  is essentially $\zeta\left(t\right)-1$, which is uniformly negative thanks to the following lemma. 
\begin{lemma}[Mean-field jump rate]\label{lm:uniform}There is a dimension-free constant $\varepsilon>0$  such that 
\begin{eqnarray*}
\PP_x\left( \zeta(t)\ge 1-\varepsilon \right) & \le & e^{-\varepsilon n},
\end{eqnarray*}
for all $x\in\Omega$ and all $t\in[1,\infty)$, provided $n\ge 2$.
\end{lemma}
\begin{proof}
Since $r$ is $[0,1]-$valued with $r(0)=0$, we clearly have
\begin{eqnarray}
\zeta(t) & \le &  1-\frac{1}{n}\sum_{j=1}^n{\bf 1}_{(X_j(t)=0)}.
\end{eqnarray}
Now, because of (\ref{assume:bounded}), we can find a region $\cI\subseteq [n]$ of size $|\cI|=\lceil n/2\rceil$ such that $\max_{i\in \cI} x_i\le 2\overline{\rho}$. Note that $|\cI^c|=\lfloor n/2\rfloor\ge n/3$, since $n\ge 2$. For each site $i\in \cI$, consider the ``good" event
\begin{eqnarray*}
G_i & := & \left\{\Xi\left([0,1]\times [0,1]\times[n]\times\{i\}\right)=0\right\}\bigcap \left\{\Xi\left([0,1]\times [0,r(1)]\times\{i\}\times \cI^c\right)\ge 2\overline{\rho}\right\}.
\end{eqnarray*}
The first part forbids any new arrival at $i$ during the time-interval $[0,1]$, while the second ensures that the $x_i\le 2\overline{\rho}$ particles will depart (recall that $r(k)\ge r(1)$ for $k\ge 1$). Thus, $G_i\subseteq \{X_i(1)=0\}$. Writing $\cP(\lambda;k)$ for the probability that a Poisson variable with mean $\lambda$ is at least $k$, we have 
\begin{eqnarray}
\label{goodevent}
\PP\left(G_i\right) & = & e^{-1}\cP\left(\frac{r(1)\lfloor n/2\rfloor}n;\lfloor 2\overline{\rho}\rfloor\right)\ \ge \ e^{-1}\cP\left(\frac{r(1)}{3};\lfloor 2\overline{\rho}\rfloor\right)=:q.
\end{eqnarray}
Since the events $(G_i)_{i\in \cI}$ are moreover independent, we conclude that the sum $\sum_{i=1}^n{\bf 1}_{(X_i(1)=0)}$ stochastically dominates a Binomial random variable with parameters $\lceil n/2\rceil$ and $q$. In particular, Hoeffding's inequality implies
\begin{eqnarray*}
\PP\left(\sum_{i=1}^n{\bf 1}_{(X_i(1)=0)}\le \frac{nq}{4} \right) & \le & \exp\left(-\frac{nq^2}{4}\right),
\end{eqnarray*}
 so that $\varepsilon=q^2/4$ satisfies the claim for $t=1$. Since the result is uniform in the choice of the initial state $x\in\Omega$, the claim automatically propagates to any time $t\ge 1$  by the Markov property. 
\end{proof}

\begin{proof}[Proof of Proposition \ref{pr:uniform}]For any observable $\varphi\colon \Omega\to\R$, the formula (\ref{dynkin}) implies that
\begin{eqnarray}
\label{dynkinexp}
\frac{{\rm d}}{{\rm d}t}\, \EE_x\left[\varphi\left(X(t)\right)\right] & = &  \EE_x\left[\left(\cL \varphi\right)\left(X(t)\right)\right].
\end{eqnarray}
We take $\varphi(y)=e^{\theta y_i}$, and we bound $\cL \varphi$ in terms of $\varphi$ to obtain a differential inequality. We have
\begin{eqnarray*}
\frac{\left(\cL \varphi\right)(y)}{\varphi(y)} & = & \left(e^{\theta}-1\right)\left\{\frac{1}{n}\sum_{j\in[n]\setminus\{i\}}\left(r\left(y_j\right)-e^{-\theta}r\left(y_i\right)\right)\right\}.
\end{eqnarray*}
The  term $\{\cdot\}$ is always less than $1$, and is even less than  $\lambda :=1-\varepsilon-e^{-\theta}r(k)$ if $y\in A\cap B$, where   
\begin{eqnarray}
A   :=  \left\{\frac{1}{n}\sum_{j=1}^nr(y_j)< 1-\varepsilon\right\} & \textrm{ and } & B:=\left\{y_i\ge k\right\}.
\end{eqnarray}
The parameters $\varepsilon\in(0,1)$ and $k\in\N$ are arbitrary for now and will be adjusted later. Thus,
\begin{eqnarray*}
{\cL \varphi} & \le & \left(e^\theta-1\right)\left(\lambda \varphi+(1-\lambda)\varphi{\bf 1}_{(A\cap B)^c}\right)\\
& \le & \left(e^\theta-1\right)\left(\lambda \varphi+2 \varphi\left({\bf 1}_{A^c}+{\bf 1}_{B^c}\right)\right)\\
 & \le &  \left(e^\theta-1\right)\left(\lambda \varphi+2e^{\theta m}{\bf 1}_{A^c}+2e^{\theta k}\right),
\end{eqnarray*}
where we have used the simple facts $\lambda\in(-1,1)$, $\|\varphi\|_\infty = e^{\theta m}$ and $\varphi{\bf 1}_{B^c}\le e^{\theta k}$.
By (\ref{dynkinexp}), we obtain
\begin{eqnarray}
\frac{{\rm d}}{{\rm d}t}\, \EE_x\left[e^{\theta X_i(t)}\right]  & \le & \left(e^\theta-1\right)\left\{\lambda \EE_x\left[e^{\theta X_i(t)}\right]+2e^{\theta m}\PP_x\left(\zeta(t)\ge 1-\varepsilon\right)+2e^{\theta k}\right\}.
\end{eqnarray}
We now choose the dimension-free constants $\theta,\varepsilon,k$ as follows. We take $\varepsilon$ as in Lemma \ref{lm:uniform}. Since $\lambda\to -\varepsilon$ as $(k,\theta)\to(\infty,0)$, we may then choose $k\in\N$ and $\theta>0$  so that $\lambda<0$. Upon further reducing $\theta$ if necessary, we may assume that $\theta  \le {\varepsilon}/\overline{\rho}$, so that $\theta m\le \varepsilon n$. For $t\ge 1$, we then have
\begin{eqnarray}
\label{ined}
\frac{{\rm d}}{{\rm d}t}\, \EE_x\left[e^{\theta X_i(t)}\right]  & \le & \alpha - \delta \EE_x\left[e^{\theta X_i(t)}\right] ,
\end{eqnarray}
where $\alpha,\delta>0$ are dimension-free constants.   It is classical that this differential inequality  implies 
\begin{eqnarray}
\EE_x\left[e^{\theta X_i(t)}\right] & \le & \frac{\alpha}{\delta}+\left(\EE_x\left[e^{\theta X_i(1)}\right]-\frac{\alpha}{\delta}\right)e^{-\delta (t-1)},
\end{eqnarray}
for $t\ge 1$. On the other hand, for $t\in[0,1]$, the domination (\ref{tightness}) implies $\EE_x\left[e^{\theta X_i(t)}\right]\le e^{\theta x+e^{\theta}-1}$. Combining these two facts, we conclude that there is a dimension-free $\kappa\in(0,\infty)$ such that 
\begin{eqnarray}
\label{ineq:unif}
\EE_x\left[e^{\theta X_i(t)}\right] & \le & \kappa\left(1+e^{\theta x_i-\delta t}\right),
\end{eqnarray}
for all $t\ge 0$. By Jensen's inequality and $(1+u)^p\le 1+u^p$, the conclusion still holds if we replace  $(\kappa,\theta,\delta)$ with $(\kappa^p,\theta p,\delta p)$ for any $p\in(0,1)$. Choosing $p$ sufficiently small so that $\kappa^p\le 2$ and $\theta p\le 1$ completes the proof of (\ref{eq:uniform}). The claim (\ref{XXX}) is then a consequence of  Chernov's bound. 
\end{proof}

\subsection{Path coupling via tagged particles}

Our proof of Theorem \ref{th:liquid} will rely on the introduction of \emph{tagged particles}. For $k\in\Z_+$, define
\begin{eqnarray}
\Delta(k) & := & r(k+1)-r(k) \ \ge \ 0.
\end{eqnarray}
Let $\Theta$ be a Poisson process of intensity $\frac{1}{n}\textrm{Leb}\otimes\textrm{Leb}\otimes\textrm{Card}$ on $\R_+\times[0,1]\times [n]$, independent of the Poisson process $\Xi$ used in the graphical construction of $\bX$, and  construct an $[n]-$valued process $\bI=(I(t)\colon t\ge 0)$ by setting $I(0)=i$ and imposing the following jumps: for each $(t,u,k)$ in $\Theta$, 
\begin{eqnarray}
\label{jumpI}
I(t) & := & 
\left\{
\begin{array}{ll}
k & \textrm{if }\Delta\left(X_{I(t-)}(t-)\right)\ge u \\
I(t-) & \textrm{else}.
\end{array}
\right.
\end{eqnarray}
In other words, conditionally on the \emph{background process} $\bX$, the \emph{tagged particle} $\bI$ performs a time-inhomogeneous random walk starting from $i$ and jumping from a site $\ell$ to a uniformly chosen site at the time-varying rate $\Delta\left(X_{\ell}(t)\right)$. 
The elementary but crucial observation is that  the process 
\begin{eqnarray}
\label{superpose}
\left(X(t)+\delta_{I(t)}\colon t\ge 0\right)
\end{eqnarray}
is then distributed as a zero-range process starting from $x+\delta_i$. Now, if $j$ is another site, we may introduce a second tagged particle $\bJ=(J(t)\colon t\ge 0)$ by setting $J(0)=j$ and for each $(t,u,k)$ in $\Theta$, 
\begin{eqnarray}
\label{jumpJ}
J(t) & := &
\left\{
\begin{array}{ll}
k & \textrm{if }\Delta\left(X_{J(t-)}(t-)\right)\ge u \\
J(t-) & \textrm{else}.
\end{array}
\right.
\end{eqnarray}
We emphasize that we use the same processes $\bX,\Theta$ to generate $\bI$ and $\bJ$. This produces two coupled zero-range processes $\left(X(t)+\delta_{I(t)}\colon t\ge 0\right)$ and $\left(X(t)+\delta_{J(t)}\colon t\ge 0\right)$  starting from $x+\delta_i$ and $x+\delta_j$. We clearly have $\{I(s)=J(s)\}\subseteq\{I(t)=J(t)\}$ for $s\le t$, and we will estimate the \emph{coalescence time}:
\begin{eqnarray}
\tau & := & \inf\left\{t\ge 0\colon I(t)=J(t)\right\} \ = \ \sup\ \left\{t\ge 0\colon I(t)\ne J(t)\right\}.
\end{eqnarray}
\begin{proposition}[Coalescence time]\label{pr:coupling}There is a dimension-free constant $\kappa<\infty$ such that %For any choice of the initial data $(x,i,j)\in\Omega\times[n]\times[n]$, 
\begin{eqnarray*}
\PP\left(\tau>\kappa\left(\|x\|_\infty\vee (\ln n)^\kappa\right)\right) & \le & \frac{\kappa}{n^2}.
\end{eqnarray*}
\end{proposition}
Let us first quickly see how this leads to Theorem \ref{th:liquid}.
\begin{proof}[Proof of Theorem \ref{th:liquid}]
Set $P_x^t=\PP_x(X(t)\in\cdot)$. By stationarity of $\pi$ and convexity of $\dtv(\cdot,\cdot)$, we have
\begin{eqnarray}
\label{convex}
\dtv\left(P_x^t, \pi\right) %& = & \dtv\left(P_t(x,\cdot), \sum_{y\in \Omega}\pi(y) P_t(y,\cdot)\right)\\
& \le & \sum_{y\in\Omega}\pi(y)\, \dtv\left(P_x^t,P_y^t\right).
\end{eqnarray}
Call $x,y\in\Omega$ \emph{adjacent}  if they differ by a single jump, i.e. 
$y=x+\delta_j-\delta_i$ for some $1\le i\ne j\le n$. When this is the case,  Proposition \ref{pr:coupling} (with $m-1$ background particles)  ensures that
 \begin{eqnarray}
t \ge \kappa\left(\|x\|_\infty\vee \|y\|_\infty\vee(\ln n)^\kappa\right) & \Longrightarrow & \dtv\left(P_x^t, P_y^t\right)  \le  \frac{\kappa}{n^2}.
 \end{eqnarray} 
Now if $x,y\in\Omega$ are arbitrary, one can always connect them by a path, i.e. a sequence $(w_0,w_1,\ldots,w_k)$  where $w_0=x$, $w_k=y$ and $w_{\ell-1}$ is adjacent to $w_{\ell}$ for $1\le \ell \le k$.  
By the triangle inequality, we have
\begin{eqnarray}
\dtv\left(P_x^t,P_y^t\right) & \le & \sum_{\ell=1}^k\dtv\left(P^t_{w_{\ell-1}},P^t_{w_{\ell}}\right).
\end{eqnarray}
Choosing a shortest path further ensures  that $k\le m$ and $\max_{1\le \ell<k}\|w_\ell\|_\infty \le \|x\|_\infty\vee\|y\|_\infty$, so that
 \begin{eqnarray}
 \label{triangle}
t \ge \kappa\left(\|x\|_\infty\vee \|y\|_\infty\vee(\ln n)^\kappa\right) & \Longrightarrow & \dtv\left(P_x^t, P_y^t\right)  \le  \frac{\kappa m}{n^2}.
 \end{eqnarray} 
In particular, if $t\ge \kappa\left(\|x\|_\infty\vee(\ln n)^\kappa\right)$, then the restriction of the sum in (\ref{convex}) to the index set $A:=\left\{y\in\Omega\colon \|y\|_\infty\le (\ln n)^\kappa\right\}$ is at most $\frac{\kappa m}{n^2}$.  On the other hand, the remaining part is at most
$
\pi(A^c) \le  e^{-\theta (\ln n)^\kappa},
$
as can be seen by taking the $t\to\infty$ limit in (\ref{XXX}). In conclusion, for any $x\in\Omega$,
 \begin{eqnarray}
t \ge \kappa\left(\|x\|_\infty\vee(\ln n)^\kappa\right) & \Longrightarrow & \dtv\left(P_x^t, \pi\right)  \le  \frac{m\kappa}{n^2}+4e^{-\theta (\ln n)^\kappa}.
 \end{eqnarray} 
Upon replacing $\kappa$ by a larger constant if necessary, we obtain the claim. 
\end{proof}
The remainder of the section is devoted to the proof of Proposition \ref{pr:coupling}. It is clear from (\ref{jumpI}),(\ref{jumpJ}) that if the two tagged particles manage to jump at the same time, then they  immediately coalesce. Note, however, that their jumps may be severely hindered by the background process: in the rate-one case for example, we have $\Delta(k)={\bf 1}_{(k=0)}$, so that the tagged particles can not jump unless they are alone! Our first step will thus consist in controlling the number of co-occupants of the tagged particles. We will then complement this estimate by showing that, when the tagged particles do not have too many co-occupants, they have a decent chance to coalesce within a short time-interval. 
\begin{lemma}[Co-occupants of the tagged particles]\label{lm:cooc}There exist dimension-free constants $\kappa_1,\kappa_2<\infty$ such that for $t=\kappa_1\|x\|_\infty$ and $a= \kappa_2\ln(1+\|x\|_\infty)$, we have
\begin{eqnarray}
\PP\left(X_{I(t)}(t)\vee Y_{J(t)}(t)\le  a\right) & \ge & \frac{1}{2}.
\end{eqnarray}
\end{lemma}
\begin{proof}
Since $I(t),J(t)$ can only move by jumps of the form (\ref{jumpI}),(\ref{jumpJ}), we necessarily have
\begin{eqnarray}
I(t),J(t) & \in & \{i,j\}\cup \left\{k\in[n]\colon \Theta\left([0,t]\times[0,1]\times\{k\}\right)\ge 1\right\}. 
\end{eqnarray}
Note that the random set on the right-hand side contains at most $2+t$ elements on average.
Taking a union bound over all these possibilities, and using the independence of $\Theta,\bX$, we obtain
\begin{eqnarray*}
\PP\left(X_{I(t)}(t)\vee X_{J(t)}(t)> a\right) & \le  & \left(2+t\right)\max_{k\in[n]}\PP\left(X_k(t)> a\right).
\end{eqnarray*}
To make this less than a half, we may choose $t=\frac{1}{\delta}\|x\|_\infty$ with $\delta$ as in (\ref{XXX}), and $a=\frac{1}{\theta}\ln(16+8t)$.
\end{proof}

\begin{lemma}[Quick coalescence]\label{lm:coal}Setting $h:=\frac{3(x_i\vee x_j)+1}{r(1)}$, we have $\PP\left(\tau \le h\right)  \ge \frac{e^{-3h}}{8}$, provided $n\ge 3$.
\end{lemma}
\begin{proof}Write $h=t+s$ with $t=3\frac{x_i\vee x_j}{r(1)}$ and $s=\frac{1}{r(1)}$, and note that $\{\tau\le h\}\supseteq G_i\cap G_j\cap F$, where
\begin{eqnarray*}
G_i & := & \left\{\Xi\left([0,t+s]\times[0,1]\times[n]\times\{i\}\right)=0\right\}\bigcap \left\{\Xi\left([0,t]\times[0,r(1)]\times \{i\}\times [n]\setminus\{i,j\}\right)\ge x_i\right\}\\
G_j & := & \left\{\Xi\left([0,t+s]\times[0,1]\times[n]\times\{j\}\right)=0\right\}\bigcap \left\{\Xi\left([0,t]\times[0,r(1)]\times \{j\}\times [n]\setminus\{i,j\}\right)\ge x_j\right\}\\
F & := & \left\{\Theta\left([0,t]\times[0,1]\times[n]\right)=0\right\}\bigcap  \left\{\Theta\left([t,t+s]\times[0,r(1)]\times[n]\right)\ge 1\right\}.
\end{eqnarray*}
Indeed, the events $G_i,G_j$ guarantee that $X_i,X_j$ are zero over the time-interval $[t,t+s]$, while
$F$ ensures that the tagged particles will remain in positions $i,j$ until time $t$, and then make an attempt to jump over $[t,t+s]$. The first such attempt will be successful for both particles, because the conditions in (\ref{jumpI}),(\ref{jumpJ}) are met (note that $\Delta(0)=r(1)$). Now,   $F,G_i,G_j$ are independent, with 
\begin{eqnarray*}
\PP\left(G_k\right) = e^{-t-s}\cP\left(\frac{(n-2)r(1)t}n; x_k\right) 
& \textrm{ and } &
\PP\left(F\right) = e^{-t}\cP\left(r(1)s;1\right),
\end{eqnarray*}
where we recall that $\cP\left(\lambda;k\right)$ is the probability that a Poisson variable with mean $\lambda$ is at least $k$.
The claim now easily follows from the classical estimate $\cP\left(\lambda;k\right)> \frac 12$, valid for any $\lambda\ge k\ge 0$. 
\end{proof}
\begin{corollary}
There is a dimension-free constant $\beta<\infty$  such that $\PP\left(\tau \le \beta \|x\|_\infty\right) \ge (1+\|x\|_\infty)^{-\beta}$.
\end{corollary}
\begin{proof}Set $t=\kappa_1\|x\|_\infty$, $a=\kappa_2\ln(1+\|x\|_\infty)$, $h=(3a+1)/r(1)$ with $\kappa_1,\kappa_2$ as in Lemma  \ref{lm:cooc}. Then,
\begin{eqnarray*}
\PP\left(\tau \le t+h\right) & \ge &\PP\left(X_{I(t)}(t)\vee X_{J(t)}(t)\le a\right)\PP\left(\left.\tau \le t+h\right |X_{I(t)}(t)\vee X_{J(t)}(t)\le a\right).
\end{eqnarray*}
The first term is at least $\frac{1}{2}$ and the second at least $\frac{1}{8}e^{-3h}$, by Lemma \ref{lm:coal} and the Markov property. 
\end{proof}
\begin{proof}[Proof of Proposition \ref{pr:coupling}]
Let $\beta$ be as in the above corollary, and let $t,a\ge 0$ be parameters to be adjusted later. Consider the increasing sequence of events $(A_k)_{k\ge 0}$ defined by 
\begin{eqnarray}
A_k & := & \left\{\tau>t+ka\beta \right\}\cap\bigcap_{\ell=0}^{k-1}\left\{\left\|X(t+\ell a \beta)\right\|_\infty\le a\right\}.
\end{eqnarray}
By the above corollary and the Markov property, we have
$
\PP\left(A_{k+1}|A_k\right)  \le   1-\left({1+a}\right)^{-\beta}.
$
Thus, 
\begin{eqnarray}
\PP\left(A_k\right) & \le & \left(1-\left({1+a}\right)^{-\beta}\right)^k \ \le \ e^{-k\left({1+a}\right)^{-\beta}}.
\end{eqnarray}
On the other hand, it is clear from the definition of $A_k$ that
\begin{eqnarray}
\PP\left(\tau>t+k a \beta \right) & \le & \PP(A_k)+k\sup_{s\ge t}\PP\left(\left\|X(s)\right\|_\infty>a\right).
\end{eqnarray}
Recalling (\ref{XXX}), we conclude that for $t=\frac{1}{\delta}\|x\|_\infty$, 
\begin{eqnarray}
\PP\left(\tau>t+ka\beta \right) & \le & e^{-k\left({1+a}\right)^{-\beta}}+4kne^{-\theta a}.
\end{eqnarray}
Choosing $a=\frac{4}{\theta}\ln n$ and $k=\lfloor (\ln n)^{2+\beta}\rfloor$ ensures that the right-hand side is $O(\frac{1}{n^2})$, as desired. 
\end{proof}

\section{Dissolution of the solid phase}
\label{sec:solid}
\subsection{Identification of the hydrodynamic limit}

With the setting of Theorem \ref{th:solid} in mind, we fix a sequence of numbers $u_1\ge u_2\ldots\ge 0$ such that
\begin{eqnarray}
\label{assume:summable}
\sum_{k=1}^\infty u_k & \le &  \rho.
\end{eqnarray}
\begin{proposition}[Resolution]\label{pr:resol}There is a unique measurable functions $f\colon\R_+\to\R_+$ satisfying 
\begin{eqnarray}
\label{recall:hydro}
f(t) & = & \int_0^t \left\{1-\Psi^{-1}\left(\rho-\sum_{k=1}^\infty \left[u_k-f(s)\right]_+\right)\right\}\dd s,
\end{eqnarray}   
for all $t\ge 0$. Moreover, $f$ is an increasing bijection from $\R_+$ to $\R_+$ and for each $i\ge 1$,
\begin{eqnarray}
\label{eq:dissolution}
f^{-1}(u_i) & = & \sum_{k=i}^\infty\frac{1}{k}\int_{\rho_k}^{\rho_{k-1}}\frac{{\rm d}s}{1-\Psi^{-1}(s)},
\end{eqnarray}
where the numbers $\rho_0\ge\rho_1\ge\ldots\ge 0$ are given by $\rho_k := \rho+ku_{k+1}-\sum_{i=1}^ku_i.$
\end{proposition}
\begin{proof}[Proof of uniqueness]
Fix $t>0$. Since $\Psi^{-1}\colon \R_+\to[0,1)$ is increasing, any solution to (\ref{recall:hydro}) satisfies
\begin{eqnarray*}
t\left(1-\Psi^{-1}(\rho)\right) & \le \ f(t)\ \le & t.
\end{eqnarray*}
Setting $\kappa(t) := \max\left\{k\ge 1\colon u_k>{t}\left(1-\Psi^{-1}(\rho)\right)\right\}$ (which is finite), we deduce that 
\begin{eqnarray}
\sum_{k=1}^\infty \left[u_k-f(t)\right]_+ & = & \sum_{k=1}^{\kappa(t)} \left[u_k-f(t)\right]_+.
\end{eqnarray}
In particular, if  $g$ is another solution to (\ref{recall:hydro}), we have
\begin{eqnarray}
\left|\sum_{k=1}^\infty \left[u_k-f(t)\right]_+-\sum_{k=1}^\infty \left[u_k-g(t)\right]_+\right| & \le & {\kappa(t)} \left|f(t)-g(t)\right|.
\end{eqnarray}
Now, $\Psi^{-1}$ is continuously differentiable and hence $\alpha-$Lipschitz on $[0,\rho]$ for some $\alpha<\infty$. Therefore, 
\begin{eqnarray}
\left|\Psi^{-1}\left(\rho- \sum_{k=1}^\infty \left[u_k-f(t)\right]_+\right)-\Psi^{-1}\left(\rho- \sum_{k=1}^\infty \left[u_k-g(t)\right]_+\right)\right| & \le & \alpha {\kappa(t)} \left|f(t)-g(t)\right|.
\end{eqnarray}
Integrating this inequality and recalling (\ref{recall:hydro}), we obtain the differential inequality
\begin{eqnarray}
\left|f(t)-g(t)\right| & \le & \alpha \int_0^t \kappa(s) \left|f(s)-g(s)\right|\dd s.
\end{eqnarray}
This will force $f=g$ by Grönwall's Lemma, provided we can show that $\kappa\in L^1(\R_+)$. But
\begin{eqnarray}
\label{integrability}
\int_0^\infty \kappa(t)\,{\rm d}t & = & \frac{1}{1-\Psi^{-1}(\rho)}\sum_{k=1}^\infty u_k,
\end{eqnarray}
by Fubini's Theorem, and the right-hand side is clearly finite, concluding the proof.
\end{proof}
\begin{proof}[Explicit resolution]  Let $\Phi\colon \R_+\to\R_+$ be the function defined by
\begin{eqnarray}
\Phi(t) & := & \int_0^{t}\frac{1}{1-\Psi^{-1}(s)}\dd s.
 \end{eqnarray} 
Note that $\Phi$ increases continuously from $0$ to $+\infty$, so that $\Phi^{-1}\colon\R_+\to\R_+$ is well-defined. Now, let $\rho=\rho_0\ge\rho_1\ge\ldots\ge 0$ and $t_1\ge t_2\ge \ldots\ge 0$ be defined by 
\begin{eqnarray}
\rho_k & := & \rho+ku_{k+1}-\sum_{i=1}^ku_i\\
t_k & := & \sum_{i=k}^\infty\frac{\Phi(\rho_{i-1})-\Phi(\rho_i)}{i}.
\end{eqnarray}
Finally, define a function $f$ separately on each $[t_{k+1},t_{k}), k\ge 0$ (with the convention $t_0=+\infty$) by
\begin{eqnarray}
\forall t\in [t_{k+1},t_{k}),\qquad f(t) & := & u_{k+1}+\frac{\Phi^{-1}\left(\Phi(\rho_k)+k(t-{t_{k+1}})\right)-\rho_k}{k}.
\end{eqnarray}
Note that for any $k\ge 1$, we have $f(t_{k})=u_k$. Moreover, the left-limit of $f$ at $t_{k}$ is 
\begin{eqnarray*}
f(t_{k}-) & = & u_{k+1}+\frac{\Phi^{-1}\left(\Phi(\rho_k)+k(t_{k}-{t_{k+1}})\right)-\rho_k}{k}\\
& = & u_{k+1}+\frac{\rho_{k-1}-\rho_k}{k}\\
& = & u_k.
\end{eqnarray*}
This shows that $f$ is continuous at each $t_k,k\ge 1$. Since $f$ is clearly continuously increasing on each $[t_k,t_{k-1})$, we deduce that $f$ is continuously increasing on the whole of $(0,\infty)$. Moreover,
\begin{eqnarray*}
f(0+) & = & \lim_{k\to \infty} \downarrow f(t_k) \ = \ \lim_{k\to \infty} \downarrow u_{k}\ = \ 0,
\end{eqnarray*}
so setting $f(0):=0$ extends $f$ into a continuously increasing function on $\R_+$. The strict monotony together with the fact that $f(t_k)=u_k$ shows that for all $t\in\R_+$ and $k\ge 1$,
\begin{eqnarray}
\label{fmono}
f(t)<u_k & \Longleftrightarrow & t<t_k.
\end{eqnarray}
Finally, $f$ is continuously differentiable on each $(t_{k+1},t_{k})$ and for $t\in (t_{k+1},t_{k})$, we easily compute
\begin{eqnarray*}
f'(t)% & = &  1-\Psi^{-1}\left(\rho_k+kf(t)-ku_{k+1}\right)\\
& = & 1-\Psi^{-1}\left(\rho+kf(t)-\sum_{i=1}^ku_i\right)\\
\\ & = & 1-\Psi^{-1}\left(\rho-\sum_{i=1}^\infty\left[u_i-f(t)\right]_+\right),
\end{eqnarray*}
where the second equality follows from (\ref{fmono}). Thus, $f$ is a solution to (\ref{recall:hydro}), and (\ref{eq:dissolution}) is clear.
\end{proof}
\subsection{Proxy for the empirical distribution}
The purpose of this section is to obtain a good approximation for the empirical measure
$
\frac{1}{n}\sum_{i=1}^n\delta_{X_i(t)}
$, out of which we will then extract a good approximation for the mean-field jump rate $\zeta(t)$. For each $z\in(0,1)$, we  define a probability distribution $\p(z)=(\p(z;k))_{k\ge 0}$ on $\Z_+$ by the formula
 \begin{eqnarray}
\p(z;k) & := & \frac{1}{R(z)}\frac{z^k}{r(1)\cdots r(k)}.
\end{eqnarray}
We extend this definition to $z=0$ by setting $\p(0)=\delta_0$.
Note that the mean of $\p(z)$ is precisely $\Psi(z)$. It will be convenient to re-parameterize $\p$ in terms of its mean by setting $\q(s):=\p\left(\Psi^{-1}(s)\right)$ for $s\in(0,\infty)$. 
We start by showing that $\q(\rho)$ is the limiting empirical distribution at equilibrium. 
\begin{lemma}[Empirical distribution at equilibrium]\label{lm:profile} In the regime (\ref{assume:sparse}), we have for any fixed $\varepsilon>0$, 
\begin{eqnarray}
\limsup_{n\to\infty}\frac 1n \log \pi\left(\left\{x\in\Omega\colon \dtv\left(\frac{1}{n}\sum_{i=1}^n\delta_{x_i},\,\q\left(\rho\right)\right)\ge\varepsilon\right\}\right)  & < & 0.
\end{eqnarray}
\end{lemma}
\begin{proof}In the degenerate case $\rho=0$, the claim is trivial since any law on $\Z_+$ is at total-variation distance at most its mean from the Dirac mass $\delta_0$. We henceforth assume that $\rho>0$, and we set $z=\Psi^{-1}(\rho)\in(0,1)$. Consider a random vector $X:=(X_1,\ldots,X_n)$ whose coordinates are i.i.d. with law $\p(z)$.
Then for any $x=(x_1,\ldots,x_n)\in\Omega$, we have
\begin{eqnarray}
\PP\left(X=x\right) %& = & %\frac{1}{\PP\left(X\in\Omega\right)}\prod_{i=1}^n\mu_z(x_i)\\
%& = & \frac{1}{\PP\left(Z\in\Omega\right)}\prod_{i=1}^n\frac{1}{R(z)}\frac{z^{x_i}}{r(1)\cdots r(x_i)}\\
& = & \frac{z^{m}}{\left(R(z)\right)^n}\prod_{i=1}^n\prod_{k=1}^{x_i}\frac{1}{r(k)}.
%\\& = & \frac{z^{m}Z\pi(x)}{\PP\left(X\in\Omega\right)\left(R(z)\right)^n},
\end{eqnarray}
Thus, $x\mapsto \PP\left(X=x\right)$ is proportional to $\pi$ on $\Omega$, and hence for any $A\subseteq \Omega$, we have the representation
\begin{eqnarray}
\label{conditional}
\pi(A) & = &  \frac{{\PP}\left(X\in A\right)}{\PP\left(X\in\Omega\right)}.
\end{eqnarray}
We now fix $\varepsilon>0$ and take
\begin{eqnarray}
\label{def:A}
A & := & \left\{x\in\Omega\colon \dtv\left(\frac{1}{n}\sum_{i=1}^n\delta_{x_i},\,\p\left(z\right)\right)\ge\varepsilon\right\}.
\end{eqnarray}
Since the coordinates of $X$ are i.i.d. with law $\p(z)$, standard large deviation estimates imply
\begin{eqnarray}
\label{sanov}
\limsup_{n\to\infty}\frac 1n\log\PP\left(X\in A\right) & < & 0;\\
\label{ldpu}
\lim_{n\to\infty}\frac{1}{n}\log \PP\left(S_n\in [0,m]\right) & = & 0;\\
\label{ldpl}
\lim_{n\to\infty}\frac{1}{n}\log \PP\left(S_n\in [m,2m]\right) & = & 0,
\end{eqnarray}
where $S_n=X_1+\cdots+X_n$. On the other hand, since $\p(z)$ is log-concave, and since this property is preserved under convolutions, the law of $S_n$ is log-concave and hence unimodal, so for any $a\in\N$,
\begin{eqnarray}
(a+1)\PP\left(S_n=m\right) & \ge & \PP\left(S_n\in[m-a,m]\right)\wedge \PP\left(S_n\in[m,m+a]\right).
\end{eqnarray}
Taking $a=m$ and using  (\ref{ldpu})-(\ref{ldpl}), we get $\frac{1}{n}\log \PP\left(X\in\Omega\right)\to 0$; (\ref{conditional})-(\ref{sanov}) completes the proof.  
\end{proof}
\begin{remark}[Monotony and regularity of $\rho\mapsto \q(\rho)$] \label{rk:lipschitz}The monotonicity (\ref{monotone}) shows that the stationary law $\pi$ is stochastically increasing in the number $m$ of particles. In view of Lemma \ref{lm:profile}, we deduce that  $\q(\rho)$ is stochastically increasing in $\rho$: if $\rho\le \rho'$, then there is a coupling $(Z,Z')$ of $\q(\rho),\q(\rho')$ such that $Z\le Z'$ almost-surely. Since $Z,Z'$ are integer-valued, we may then write
\begin{eqnarray}
\dtv\left(\q(\rho),\q(\rho')\right) & \le &  \EE\left[|Z'-Z|\right] \ = \ \EE[Z']-\EE[Z] \ = \ \rho'-\rho.
\end{eqnarray}
In conclusion, $\rho\mapsto \q(\rho)$ is increasing and $1-$Lipschitz. 
\end{remark}
We will show that the approximation $\frac{1}{n}\sum_{i=1}^n\delta_{X_i}\approx \q(\rho)$ remains valid out of equilibrium, provided $\rho$ is replaced by an \emph{effective density}, obtained by ignoring the particles in the solid phase. To formalize this, we assume that the solid phase is initially restricted to some fixed region $\{1,\ldots,L\}$:
\begin{eqnarray}
\label{assume:finitary}
\max_{L<i\le n}x_i & = & o(n).
\end{eqnarray}
Note that by (\ref{XXX}), this property is preserved by the dynamics in the sense that for any fixed $t\ge 0$,
\begin{eqnarray}
\label{cons:finitary}
\max_{L< i\le n}X_i(nt) & = & o(n),
\end{eqnarray}
almost-surely (as long as all processes live on the same probability space). 
\begin{proposition}[Proxy for the empirical measure]\label{pr:finitary}If $x=x(n)$ satisfies (\ref{assume:finitary}), then for fixed $t>0$,
\begin{eqnarray}
\EE_x\left[\dtv\left(\frac{1}{n}\sum_{i=1}^n\delta_{X_i(nt)},\, \q\left(\frac 1n\sum_{i=L+1}^n{X_i(nt)}\right)\right)\right] & \xrightarrow[n\to\infty]{} & 0.
\end{eqnarray}
Since the rate function $r$ has mean $z$ under the law $\p(z)$, we have in particular
\begin{eqnarray}
\label{rate}
\EE_x\left[\left|\zeta(nt)-\Psi^{-1}\left(\frac mn-\frac 1n\sum_{i=1}^L{X_i(nt)}\right)\right|\right] & \xrightarrow[n\to\infty]{} & 0.
\end{eqnarray}
\end{proposition}

Our proof will consist in comparing the system with one where the solid phase is removed, so that Theorem \ref{th:liquid} becomes applicable. We will rely on the following  lemma. 
\begin{lemma}[Truncation]\label{lm:truncation}Fix $x\in\Z_+^n$, and let $\widehat{x}$ be obtained by zeroing the first $L$ coordinates. Then, the processes $\bf X,\widehat{X}$ obtained by applying the graphical construction  to $x,\widehat{x}$ satisfy
\begin{eqnarray*}
\EE\left[\dtv\left(\frac{1}{n}\sum_{i=1}^n\delta_{X_i(t)},\, \frac{1}{n}\sum_{i=1}^n\delta_{\widehat{X}_i(t)}\right) \right] & \le & \frac{L(1+t)}n.
\end{eqnarray*}

\end{lemma}
\begin{proof}
By (\ref{monotone}), we have $\widehat{X}(t)  \le  X(t)$ for all $t\ge 0$. In particular, 
\begin{eqnarray*}
\sum_{i=L+1}^n\left|X_i(t)-\widehat{X}_i(t)\right| & = & \sum_{i=L+1}^nX_i(t)-\sum_{i=L+1}^n\widehat{X}_i(t).
\end{eqnarray*}
Now, observe that the right-hand side equals zero when $t=0$, and that the only jumps of the form (\ref{jump}) that may increment it (by $1$ unit each time) are those whose source $i$ is in $[L]$. Consequently,
 \begin{eqnarray}
 \label{lips}
\EE\left[\sum_{i=L+1}^n\left|\widehat{X}_i(t)-X_i(t)\right|\right] & \le & \EE\left[\Xi\left(\left[0,t\right]\times[0,1]\times [L]\times [n]\right)\right] \ = \ Lt.
\end{eqnarray}
On the other hand, by definition of the total-variation distance, we have
\begin{eqnarray*}
\dtv\left(\frac{1}{n}\sum_{i=1}^n\delta_{X_i(t)},\, \frac{1}{n}\sum_{i=1}^n\delta_{\widehat{X}_i(t)}\right) %& = & \sup_{A\subseteq \Z_+}\left|\frac{1}{n}\sum_{i=1}^n{\bf 1}_{(X_i(t)\in A)}-\frac{1}{n}\sum_{i=1}^n{\bf 1}_{(\widehat{X}_i(t)\in A)}\right|\\ 
&  \le &  \frac{1}{n}\sum_{i=1}^n{\bf 1}_{\left(X_i(t)\ne\widehat{X}_i(t)\right)}.
\end{eqnarray*}
To conclude, we simply bound ${\bf 1}_{\left(X_i(t)\ne\widehat{X}_i(t)\right)}$ by $1$ for $i\le L$, and by $|X_i(t)-\widehat{X}_i(t)|$ for $i>L$. 
\end{proof}

\begin{proof}[Proof of proposition \ref{pr:finitary}]
If $L=0$, then $\|x\|=o(n)$, so Theorem \ref{th:liquid} ensures that for fixed $t>0$,
\begin{eqnarray}
\max_{A\subseteq \Omega}\left|\PP_x\left(X(nt)\in A\right) - \pi(A)\right|  & \xrightarrow[n\to\infty]{} & 0
\end{eqnarray}
Since the event $A$ in Lemma \ref{lm:profile} satisfies  $\pi(A)\to 0$, we must have $\PP_x\left(X(nt)\in A\right)\to 0$. Thus,
\begin{eqnarray*}
\EE_x\left[\dtv\left(\frac{1}{n}\sum_{i=1}^n\delta_{X_i(nt)},\,\q\left(\rho\right)\right)\right] & \xrightarrow[n\to\infty]{} & 0.
\end{eqnarray*}
On the other hand, we have $\dtv\left(\q\left(\frac m n\right),\q(\rho)\right)\to 0$, so the case $L=0$ is proved. 
Now, assume that $x$ satisfies (\ref{assume:finitary}) for some $L\ge 1$, and let $\widehat{x}$ be as  in Lemma \ref{lm:truncation}. Then  $\|\widehat{x}\|_\infty=o(n)$ by construction, so the case $L=0$ with $\widehat{m}:=m-(x_1+\cdots+x_L)$ particles instead of $m$ implies
\begin{eqnarray*}
\EE\left[\dtv\left(\frac{1}{n}\sum_{i=1}^n\delta_{{\widehat{X}}_i(nt)},\,\q\left(\frac{\widehat{m}}{n}\right)\right)\right] & \xrightarrow[n\to\infty]{} & 0.
\end{eqnarray*}
On the other hand, under the coupling of Lemma \ref{lm:truncation}, we have 
\begin{eqnarray*}
\EE\left[\dtv\left(\frac{1}{n}\sum_{i=1}^n\delta_{\widehat{X}_i(nt)},\,\frac{1}{n}\sum_{i=1}^n\delta_{{X}_i(nt)}\right)\right] & \le & L\left(t+\frac{1}{n}\right).
\end{eqnarray*}
Finally, Remark \ref{rk:lipschitz} implies that
\begin{eqnarray*}
\dtv\left(\q\left(\frac{{\widehat{m}}}{n}\right),\,\q\left(\sum_{i=L+1}^n\frac{X_i(nt)}{n}\right)\right) & \le & \frac{1}{n}\left|\sum_{i=1}^L{X_i(nt)}-\sum_{i=1}^L{x_i}\right|,
\end{eqnarray*}
and the right-hand side has mean at most $2Lt$ by (\ref{tightness}). 
By the triangle inequality, we conclude that
\begin{eqnarray*}
\limsup_{n\to\infty} \EE_x\left[\dtv\left(\frac{1}{n}\sum_{i=1}^n\delta_{X_i(nt)},\,\q\left(\frac 1n\sum_{i=L+1}^n{X_i(nt)}\right)\right)\right] & \le & 3Lt.
\end{eqnarray*}
This may seem rather weak compared to what we want to establish. However, by (\ref{cons:finitary}), we may apply this result with $x$ replaced by  $X(ns)$ and then invoke the Markov property to obtain
 \begin{eqnarray*}
\limsup_{n\to\infty} \EE_x\left[\dtv\left(\frac{1}{n}\sum_{i=1}^n\delta_{X_i(ns+nt)},\,\q\left(\frac 1n\sum_{i=L+1}^n{X_i(ns+nt)}\right)\right)\right] & \le & 3Lt,
\end{eqnarray*}
for any $s \ge 0$ and $t>0$. Replacing $t$ with $\varepsilon$ and $s$ with $t-\varepsilon$, we see that for any $0< \varepsilon\le t$,
 \begin{eqnarray*}
\limsup_{n\to\infty} \EE_x\left[\dtv\left(\frac{1}{n}\sum_{i=1}^n\delta_{X_i(nt)}-\q\left(\frac 1n\sum_{i=L+1}^n{X_i(nt)}\right)\right)\right] & \le & 3L\varepsilon.
\end{eqnarray*}
Since $\varepsilon$ can be made arbitrarily small, the result follows.
\end{proof}

\subsection{Tightness and convergence}
We are now ready to prove Theorem \ref{th:solid}, using the classical tightness-uniqueness strategy. Define
\begin{eqnarray}
U_i^n(t) \ := \ \frac{X_i(nt)}{n}  & \textrm{ and } & V^n(t) \ := \ \int_0^t\left(1-\zeta(ns)\right)\dd s.
\end{eqnarray}
The fact that $U_i^n(t)\in[0,\overline{\rho}]$ and the domination (\ref{tightness}) suffice to guarantee the tightness of $\left(U^n_i\right)_{n\ge 1}$ in the Skorokhod space $D(\R_+,\R)$,  and the continuity of any weak sub-sequential limit $U^\star_i$. The same conclusion applies to $(V^n)_{n\ge 1}$, because $\zeta$ is $[0,1]-$valued. Our objective is to show that necessarily,
\begin{eqnarray}
\label{goalsko}
U_i^\star & = & \left[u_i-f\right]_+.
\end{eqnarray}
By diagonal extraction, we may find a sub-sequence along which we have the joint convergence
\begin{eqnarray}
\label{extract}
\left(V^n,U_1^n,U_2^n,\ldots\right) & \xrightarrow[]{} & \left(V^\star,U_1^\star,U_2^\star,\ldots\right),
\end{eqnarray}
with respect to the product topology. 
By Skorokhod's Theorem, we may even assume for convenience that the convergence (\ref{extract}) is almost-sure. Our plan is to pass to the limit in the martingale
\begin{eqnarray}
\label{Mi}
M^n_i(t) & := & U_i^n(t) - U_i^n(0)+V^n(t)-\int_0^t \left(1-r\left(nU_i^{n}(s)\right)\right)\dd s,
\end{eqnarray}
which is just a rescaled version of (\ref{dynkin:single}). Since $U_i^n$ has jumps of size at most $\frac 1n$ occuring at rate at most $2n$, a classical exponential concentration estimate for martingales (see, e.g., \cite{MR838963}) ensures that
\begin{eqnarray}
%\label{doob}
%\PP\left(\sup_{s\in[0,t]} |M_i^n(s)|\ge \varepsilon\right) & \le & 2\exp\left(-\frac{n \varepsilon^2}{2\varepsilon+4t}%\right),
M_i^n(t) & \xrightarrow[n\to\infty]{} & 0,
 \end{eqnarray} 
almost-surely. On the other hand, (\ref{XXX}) easily imply that for fixed $t,h\ge 0$, 
\begin{eqnarray}
\label{drift}
\limsup_{n\to\infty}\max_{i\in[n]}\left\{U_i^n(t+h)-\left[{U_i^n(t)}-\delta h\right]_+\right\} & \le & 0,
\end{eqnarray}
which shows in particular that $U_i^\star(t+h)\le \left[U_i^\star(t)-\delta h\right]_+$. Thus, $U_i^\star$ is non-increasing.  Consequently, on the event $\{U_i^\star(t)>0\}$, we have $U_i^\star(s)>0$ for all $s\le t$, and hence $r(nU_i^n(s))\to 1$, by our assumption (\ref{assume:rates}). Passing to the limit in (\ref{Mi}), we conclude that the equality
\begin{eqnarray}
\label{lim}
U_i^\star(t) & = & u_i-V^\star(t),
\end{eqnarray}
holds as long as $U_i^\star(t)>0$. But both sides of (\ref{lim}) are continuous and non-increasing, so they must reach zero at the same time. Since the left-hand side is non-negative, we conclude that  for all $t\ge 0$.
\begin{eqnarray}
U_i^\star(t) & = & \left[u_i-V^\star(t)\right]_+.
\end{eqnarray}
 Comparing this with (\ref{goalsko}), we now only have to show that $V^\star$ solves (\ref{recall:hydro}), i.e.
\begin{eqnarray}
\label{newgoal}
V^\star(t) & = & \int_0^t\left\{1-\Psi^{-1}\left(\rho-\sum_{i=1}^\infty{U_i^\star(s)}\right)\right\}\dd s.
\end{eqnarray}
Fix a non-negative integer $L$. Taking $t=0$ in (\ref{drift}), we deduce that 
\begin{eqnarray}
\label{max}
\limsup_{n\to\infty}\max_{L<i\le n}{U_i^n(h)} & \le & \left[u_{L+1}-\delta h\right]_+. 
\end{eqnarray}
Choosing $h=\frac{u_{L+1}}{\delta}$ ensures that the right-hand side is zero. Consequently, we may apply (\ref{rate}) with the initial state being $X(nh)$ and use Markov's property to obtain that for any fixed $t>h$.
\begin{eqnarray}
\label{first}
\EE_x\left[\left|\zeta(nt)-\Psi^{-1}\left(\frac mn-\sum_{i=1}^L{U_i^n(t)}\right)\right|\right] & \xrightarrow[n\to\infty]{} & 0.
\end{eqnarray}
On the other hand, by continuity of $\Psi^{-1}$, we have almost-surely,
\begin{eqnarray}
\Psi^{-1}\left(\frac mn-\sum_{i=1}^L{U_i^n(t)}\right) & \xrightarrow[n\to\infty]{} & \Psi^{-1}\left(\rho-\sum_{i=1}^L{U_i^\star(t)}\right).
\end{eqnarray}
Moreover, we can safely replace $L$ by $\infty$ on the right-hand side, because (\ref{max}) ensures that $U_i^\star(t)=0$ for all  $i>L$. Combining this with (\ref{first}), we arrive at
 \begin{eqnarray}
 \label{rate*}
\EE_x\left[\left|\zeta(nt)-\Psi^{-1}\left(\rho-\sum_{i=1}^\infty{U_i^\star(t)}\right)\right|\right] & \xrightarrow[n\to\infty]{} & 0.
\end{eqnarray}
This is true for any $t>h$, but $h=\frac{u_{L+1}}{\delta}$ can be made arbitrarily small by choosing $L$ large, so (\ref{rate*}) holds for any $t>0$. Integrating over $t$, we easily deduce (\ref{newgoal}). Finally, note that the convergence
\begin{eqnarray}
\label{conclude}
{U_i^n} & \xrightarrow[n\to\infty]{} & \left[u_i-f\right]_+
\end{eqnarray}
is automatically uniform on compact subsets of $\R_+$, because the limit is continuous. It is also uniform in $i$, because (\ref{max}) ensures that $\max_{L<i\le n}{U_i^n(h)}$ can be made arbitrarily small by choosing $L$ large enough, uniformly in $n$. This concludes the proof of Theorem \ref{th:solid}.

%XXX ADD REFS: LLT, ETHIER-KURTZ, ZRP CINETICS, CUTOFF OTHER PARTICLE SYSTEMS.
\bibliographystyle{plain}
\bibliography{ZRP}

\begin{thebibliography}{10}

\bibitem{MR770418}
David Aldous.
\newblock Random walks on finite groups and rapidly mixing {M}arkov chains.
\newblock In {\em Seminar on probability, {XVII}}, volume 986 of {\em Lecture
  Notes in Math.}, pages 243--297. Springer, Berlin, 1983.

\bibitem{MR2200172}
Anne-Severine Boudou, Pietro Caputo, Paolo Dai~Pra, and Gustavo Posta.
\newblock Spectral gap estimates for interacting particle systems via a
  {B}ochner-type identity.
\newblock {\em J. Funct. Anal.}, 232(1):222--258, 2006.
\newblock \href{http://www.ams.org/mathscinet-getitem?mr=MR2200172}{MR2200172}.

\bibitem{Pathcoupling}
Russ Bubley and Martin Dyer.
\newblock Path coupling: A technique for proving rapid mixing in markov chains.
\newblock In {\em Proceedings of the 38th Annual Symposium on Foundations of
  Computer Science}, FOCS '97, pages 223--, Washington, DC, USA, 1997. IEEE
  Computer Society.

\bibitem{MR2073330}
Pietro Caputo.
\newblock Spectral gap inequalities in product spaces with conservation laws.
\newblock In {\em Stochastic analysis on large scale interacting systems},
  volume~39 of {\em Adv. Stud. Pure Math.}, pages 53--88. Math. Soc. Japan,
  Tokyo, 2004.
\newblock \href{http://www.ams.org/mathscinet-getitem?mr=MR2073330}{MR2073330}.

\bibitem{Cap}
Pietro Caputo and Gustavo Posta.
\newblock Entropy dissipation estimates in a zero-range dynamics.
\newblock {\em Probab. Theory Related Fields}, 139(1-2):65--87, 2007.
\newblock \href{http://www.ams.org/mathscinet-getitem?mr=MR2322692}{MR2322692}.

\bibitem{MR2184099}
Paolo Dai~Pra and Gustavo Posta.
\newblock Logarithmic {S}obolev inequality for zero-range dynamics.
\newblock {\em Ann. Probab.}, 33(6):2355--2401, 2005.
\newblock \href{http://www.ams.org/mathscinet-getitem?mr=MR2184099}{MR2184099}.

\bibitem{MR1374011}
Persi Diaconis.
\newblock The cutoff phenomenon in finite {M}arkov chains.
\newblock {\em Proc. Nat. Acad. Sci. U.S.A.}, 93(4):1659--1664, 1996.

\bibitem{MR838085}
Stewart~N. Ethier and Thomas~G. Kurtz.
\newblock {\em Markov processes}.
\newblock Wiley Series in Probability and Mathematical Statistics: Probability
  and Mathematical Statistics. John Wiley \& Sons, Inc., New York, 1986.
\newblock Characterization and convergence.

\bibitem{Evans}
M.~R. Evans and T.~Hanney.
\newblock Nonequilibrium statistical mechanics of the zero-range process and
  related models.
\newblock {\em J. Phys. A}, 38(19):R195--R240, 2005.
\newblock \href{http://www.ams.org/mathscinet-getitem?mr=MR2145800}{MR2145800}.

\bibitem{MR3513606}
Max Fathi and Jan Maas.
\newblock Entropic {R}icci curvature bounds for discrete interacting systems.
\newblock {\em Ann. Appl. Probab.}, 26(3):1774--1806, 2016.

\bibitem{Graham}
Benjamin~T. Graham.
\newblock Rate of relaxation for a mean-field zero-range process.
\newblock {\em Ann. Appl. Probab.}, 19(2):497--520, 2009.
\newblock \href{http://www.ams.org/mathscinet-getitem?mr=MR2521877}{MR2521877}.

\bibitem{HS}
Jonathan {Hermon} and Justin {Salez}.
\newblock {A version of Aldous' spectral-gap conjecture for the zero range
  process}.
\newblock {\em ArXiv e-prints}, August 2018.

\bibitem{MR1681098}
E.~Janvresse, C.~Landim, J.~Quastel, and H.~T. Yau.
\newblock Relaxation to equilibrium of conservative dynamics. {I}. {Z}ero-range
  processes.
\newblock {\em Ann. Probab.}, 27(1):325--360, 1999.
\newblock \href{http://www.ams.org/mathscinet-getitem?mr=MR1681098}{MR1681098}.

\bibitem{Lacoincycle}
Hubert Lacoin.
\newblock The cutoff profile for the simple exclusion process on the circle.
\newblock {\em Ann. Probab.}, 44(5):3399--3430, 2016.
\newblock \href{http://www.ams.org/mathscinet-getitem?mr=MR3551201}{MR3551201}.

\bibitem{Lacoin}
Hubert Lacoin.
\newblock Mixing time and cutoff for the adjacent transposition shuffle and the
  simple exclusion.
\newblock {\em Ann. Probab.}, 44(2):1426--1487, 2016.
\newblock \href{http://www.ams.org/mathscinet-getitem?mr=MR3474475}{MR3474475}.

\bibitem{MR1415232}
C.~Landim, S.~Sethuraman, and S.~Varadhan.
\newblock Spectral gap for zero-range dynamics.
\newblock {\em Ann. Probab.}, 24(4):1871--1902, 1996.
\newblock \href{http://www.ams.org/mathscinet-getitem?mr=MR1415232}{MR1415232}.

\bibitem{MR3726904}
David~A. Levin, Yuval Peres, and Elizabeth~L. Wilmer.
\newblock {\em Markov chains and mixing times}.
\newblock American Mathematical Society, Providence, RI, 2017.
\newblock Second edition of [ MR2466937], With a chapter on ``Coupling from the
  past'' by James G. Propp and David B. Wilson.

\bibitem{liggettbook2}
Thomas~M. Liggett.
\newblock {\em Stochastic interacting systems: contact, voter and exclusion
  processes}, volume 324 of {\em Grundlehren der Mathematischen Wissenschaften
  [Fundamental Principles of Mathematical Sciences]}.
\newblock Springer-Verlag, Berlin, 1999.
\newblock \href{http://www.ams.org/mathscinet-getitem?mr=MR1717346}{MR1717346}.

\bibitem{liggettbook1}
Thomas~M. Liggett.
\newblock {\em Interacting particle systems}.
\newblock Classics in Mathematics. Springer-Verlag, Berlin, 2005.
\newblock Reprint of the 1985 original.

\bibitem{MS}
Mathieu Merle and Justin Salez.
\newblock Cutoff for the mean-field zero-range process.
\newblock {\em arXiv preprint arXiv:1804.04608}, 2018.

\bibitem{MR2341319}
Ravi Montenegro and Prasad Tetali.
\newblock Mathematical aspects of mixing times in {M}arkov chains.
\newblock {\em Found. Trends Theor. Comput. Sci.}, 1(3):x+121, 2006.
\newblock \href{http://www.ams.org/mathscinet-getitem?mr=MR2341319}{MR2341319}.

\bibitem{MorrisZRP}
Ben Morris.
\newblock Spectral gap for the zero range process with constant rate.
\newblock {\em Ann. Probab.}, 34(5):1645--1664, 2006.
\newblock \href{http://www.ams.org/mathscinet-getitem?mr=MR2271475}{MR2271475}.

\bibitem{MR838963}
Galen~R. Shorack and Jon~A. Wellner.
\newblock {\em Empirical processes with applications to statistics}.
\newblock Wiley Series in Probability and Mathematical Statistics: Probability
  and Mathematical Statistics. John Wiley \& Sons, Inc., New York, 1986.

\bibitem{Spitzer}
Frank Spitzer.
\newblock Interaction of {M}arkov processes.
\newblock {\em Advances in Math.}, 5:246--290 (1970), 1970.
\newblock \href{http://www.ams.org/mathscinet-getitem?mr=MR0268959}{MR0268959}.

\end{thebibliography}

\end{document}